\documentclass[a4paper,12pt]{amsart}
\usepackage{mathtools}
\usepackage{amsmath, amssymb, amscd, amsthm, amsfonts, amsbsy}
\usepackage{array}
\usepackage{epsfig}
\usepackage{ragged2e}
\usepackage{graphics}
\usepackage{fancyhdr}
\usepackage[all,cmtip]{xy} 
\usepackage{ifthen}
\usepackage[pdftex,bookmarks=true]{hyperref}
\nonstopmode \numberwithin{equation}{section}
\makeatletter
\newcommand{\ubrace}[2]{\mathord{\mathpalette\ubrace@{{#1}{#2}}}}
\newcommand{\ubrace@}[2]{\ubrace@@#1#2}
\newcommand{\ubrace@@}[3]{
	\underbrace{#1#2}_{#3}%
}
\makeatother

\newtheorem{Th}{Theorem}[subsection]
\newtheorem{cor}[Th]{Corollary}

\theoremstyle{definition}
\newtheorem{defn}[Th]{Definition}

\newtheorem{ex}[Th]{Example}

\newtheorem{Thm}{Theorem}[section]

\newtheorem{Lem}[Thm]{Lemma}
\newtheorem{Prop}[Thm]{Proposition}

\theoremstyle{definition}

\newtheorem{Rem}[Thm]{Remark}
\newtheorem{Ex}[Thm]{Example}
\begin{document}
	\bibliographystyle{amsplain}
	\title{A structure theorem and left-orderability of a quotient of quasi-isometry group of the real line}
	\author{Swarup Bhowmik}
	\address{Swarup Bhowmik, Department of Mathematics,
		Indian Institute of Technology Kharagpur, Kharagpur - 721302, India.}
	\email{swarup.bhowmik@iitkgp.ac.in}
	\author{Prateep Chakraborty}
	\address{Prateep Chakraborty, Department of Mathematics,
		Indian Institute of Technology Kharagpur, Kharagpur - 721302, India.}
	\email{prateep@maths.iitkgp.ac.in}
	\begin{abstract}
	It is well-known that $QI(\mathbb{R})\cong(QI(\mathbb{R}_{+})\times QI(\mathbb{R}_{-}))\rtimes <t>$, where $QI(\mathbb{R})$(resp. $QI(\mathbb{R}_{+})(\cong QI(\mathbb{R_-}))$) is the group of quasi-isometries of the real line (resp. $[0,\infty)$). We introduce an invariant for the elements of $QI(\mathbb{R_{+}})$ and split it into smaller units. We give an almost characterization of the elements of these units. We also show that a quotient of $QI(\mathbb{R_{+}})$ gives an example of a left-orderable group which is not locally indicable.
	\end{abstract}
	\thanks{The first author of this article acknowledges the financial support from Inspire, DST, Govt. of India as a Ph.D. student (Inspire) sanction
		letter number: No. DST/INSPIRE Fellowship/2018/IF180972}
	\maketitle
	\pagestyle{myheadings}
	\markboth{S. Bhowmik and P. Chakraborty}{A structure theorem and left-orderablity of a quotient of quasi-isometry group of the real line}
	\bigskip
	\section{Introduction}
	The notion of quasi-isometry is one of the fundamental concepts in geometric group theory. Let $(X,d)$ be a metric space. The group of quasi-isometries from $X$ to itself is denoted by $QI(X)$ and is a quasi-isometric invariant of $X$. It is well-knowm that when $\Gamma$ is a finitely generated group, a choice of finite generating set $S$ gives the word metric $d_S$ on $\Gamma$, making it a metric space. Then one can talk about the quasi-isometry group $QI(\Gamma)$ and it can be shown that $QI(\Gamma)$ does not depend on the choice of $S$. Hence $QI(\Gamma)$ becomes an invariant of the group $\Gamma$. In general, $QI(\Gamma)$ is hard to determine and for very few families of groups $\Gamma$, $QI(\Gamma)$ has been explicitly studied, for example, solvable Baumslag-Solitar groups $BS(1,n)$ \cite[Theorem 7.1]{Farb Mosher}, the groups $BS(m,n),~1<m<n$ \cite[Theorem 4.3]{Whyte}, irreducible lattices in semisimple Lie groups (see \cite{Farb} and the references therein) etc.\\
	Even for $\Gamma=\mathbb{Z}^n$, $QI(\mathbb{Z}^n)(\cong QI(\mathbb{R}^n))$ remain largely unexplored, especially for $n>1$, though some attempts have been made in \cite{Mitra Sankaran} and \cite{Bhowmik Chakraborty}. In particular when $n=1$, Gromov and Pansu \cite[\textsection 3.3.B]{Gromov Pansu} observed that $Bilip(\mathbb{R})\rightarrow QI(\mathbb{R})$ is surjective and that $QI(\mathbb{R})$ is infinite dimensional . Furthermore, Sankaran \cite{Sankaran} proved that there is a surjection from $PL_\delta(\mathbb{R})$ to $QI(\mathbb{R})$. This work of Sankaran gives a nice and well understood representation of the elements of $QI(\mathbb{R})$. However, it is not enough to say when two such representations lie in the same equivalence class. This article is an attempt to characterize the elements of $QI(\mathbb{R})$ and split $QI(\mathbb{R})$ into smaller and nicer units. \\
	It is already observed in \cite{Ye Zhao} that $QI(\mathbb{R})\cong(QI(\mathbb{R}_{+})\times QI(\mathbb{R}_{-}))\rtimes <t>$, where $QI(\mathbb{R}_+)$ (resp. $QI(\mathbb{R}_-)$) is the quasi-isometry group of the ray [0,$\infty$) (resp. (-$\infty$,0]) viewed as subgroup of $QI(\mathbb{R})$ fixing the negative part (resp. positive part) and $t\in QI(\mathbb{R})$ is the reflection $t(x)=-x$ for any $x\in\mathbb{R}$. So to characterize $QI(\mathbb{R})$, it is enough to study $QI(\mathbb{R}_{+})$ due to the fact that $QI(\mathbb{R_{+}})\cong QI(\mathbb{R_-})$. \\ In this article, we introduce an invariant defined by\\ $S_{[g]}=\Big\{\displaystyle\lim_{n\rightarrow\infty}\frac {g(x_n)}{x_n}:x_n\rightarrow\infty~\text{and}~ \frac {g(x_n)}{x_n}~\text{converges}\Big\}$ for any $[g]\in QI(\mathbb{R}_{+})$. With respect to this invariant, we define some subsets of $QI(\mathbb{R}_{+})$ as follows: \\
	For $c>0$ and $I\subset (0,\infty)$, $H_c=\{[g]\in QI(\mathbb{R_{+}}): S_{[g]}=c\}$, $H_I=\{[g]\in QI(\mathbb{R_{+}}): S_{[g]}=I\}$ and $\overline{H_c}=\displaystyle\bigcup_{c\in\mathbb{R}_{>0}}H_c$. Note that for $c=1$, $H_c$ coincides with $H$, already defined in \textsection 2.1 of \cite{Ye Zhao}. Then we show that 	$QI(\mathbb{R_{+}})/\overline{H_c}=\displaystyle\bigsqcup_{M\geq 1}H_{[1,M]}/H$ where the above decomposition is that of right cosets of $QI(\mathbb{R_{+}})/\overline{H_c}$. Then, to explore the set $H_{[1,M]}/H$, we obtain an almost complete characterization of its elements in terms of the slopes of the representations coming from $PL_\delta(\mathbb{R})$.\\
	 For $[f],[g]\in H_{[1,M]}/H,$ if no sequence $\{a_n\}$ exists such that $\{a_n\}\rightarrow\infty$ and  $\displaystyle\lim_{n\rightarrow\infty}\frac {f(a_n)}{a_n}=1=\displaystyle\lim_{n\rightarrow\infty}\frac {g(a_n)}{a_n}$, then $H[f]\neq H[g]$. On the other hand, suppose, there exists a strictly increasing sequence $\{a_n\}$ with $\displaystyle\lim_{n\rightarrow\infty}\frac {f(a_n)}{a_n}=1=\displaystyle\lim_{n\rightarrow\infty}\frac {g(a_n)}{a_n}$. Then on each interval $[a_n,a_{n+1}]$, we get a partition $a_n=x_{n,0}<x_{n,1}<x_{n,2}<...<x_{n,r}=a_{n+1}$
	 such that on each interval $[x_{n,i-1},x_{n,i}]$, $f$ and $g$ are both linear. So, for each $n$ we get two $r$-tuples $(\lambda_{n,1},\lambda_{n,2},...,\lambda_{n,r})$ and $(\lambda_{n,1}^{'},\lambda_{n,2}^{'},...,\lambda_{n,r}^{'})$
	 where $\lambda_{n,i}$ and $\lambda_{n,i}^{'}$ are slopes of $f$ and $g$ on $[x_{n,i-1},x_{n,i}]$ respectively. Thus two sequences $\{A_n\}$ and $\{B_n\}$ appear in $\mathbb{R}_{>0}^{\infty}\big(=\displaystyle\bigcup_{m\in\mathbb{N}}\mathbb{R}_{>0}^{m}\big)$ for $f$ and $g$ respectively where $A_n$ and $B_n$ are tuples of slopes of $f$ and $g$ on $[a_n,a_{n+1}]$ respectively.
	 \begin{Thm}\label{1.1}
	 	\textit{Suppose $H[f],H[g]\in H_{[1,M]}/H$ and there exists a sequence $\{a_n\}$ with $\frac {f(a_n)}{a_n},\frac {g(a_n)}{a_n}\rightarrow 1.$ Then\\
	 	(a) if $||A_n-B_n||_{\infty}\rightarrow 0$, then $H[f]=H[g]$ and\\ 
	 	(b) if there exists a $K_1>1$ such that $x_{n,i}>K_1x_{n,i-1}$ and $H[f]=H[g]$, then $||A_n-B_n||_{\infty}\rightarrow 0$.}
	 \end{Thm} 
    In recent times left-orderability and locally indicability of a group are drawing the attention of many researchers since the left-orderable groups, left-invariant orders on groups and locally indicable groups have strong connections with algebra, dynamics and topology. It is known that every locally indicable group is left orderable \cite[Theorem 7.3.11]{Mura Rhemtulla}; it was an open question whether the converse was true. In \cite{Bergman}, the author came up with a counterexample. This note provides another counterexample in the form of $QI(\mathbb{R_{+}})/H$ and attains the following results.
    \begin{Thm}\label{1.2}
    	\textit{The quasi-isometry group $QI(\mathbb{R_{+}})/H$ is left-orderable.}
    \end{Thm}
\begin{Thm}\label{1.3}
	\textit{The quasi-isometry group $QI(\mathbb{R_{+}})/H$ is not locally indicable.}
\end{Thm}
	In another direction, as an application of Theorem \ref{1.1}, we obtain a large class of non-commuting elements of the quasi-isometry group $QI(\mathbb{R}_{+})$.
	\section{Preliminaries}
	\subsection{Quasi-isometry}
	We begin by describing the notion of quasi-isometry. Let $f:(X,d)\rightarrow (X',d')$ be a map between metric spaces. We say that $f$ is a quasi-isometric embedding if there exists an $M>1$ such that 
	\begin{center}
		$\frac {1}{M}d(x,y)-M\leq d'(f(x),f(y))\leq Md(x,y)+M$, for all $x,y\in X$.
	\end{center}
	If in addition, the image of $f$ is $M'$-dense, that is, if there exists a constant $M'$ such that for any $x'\in X'$, there exists $x\in X$ such that $d'(f(x),x')<M'$, then $f$ is said to be a quasi-isometry. In general, a quasi-isometry $f:X\rightarrow X$ is neither one-one nor onto. So, it would not be possible to find an inverse for $f$. However, there exists another quasi-isometry $g:X'\rightarrow X$, called a quasi-inverse of $f$ such that $g\circ f$ (resp. $f\circ g$) is quasi-isometrically equivalent to the identity map of $X$ (reps. $X'$). (Two maps $f,g:X\rightarrow X'$ are said to be quasi-isometrically equivalent if there exists $C>0$ such that $d'(f(x),g(x))<C$, for all $x\in X$.)\\
	We denote the equivalence class of a quasi-isometry $f:X\rightarrow X$ by $[f]$. Though sometimes $f$ is also used instead of $[f]$ for the convenience of notation. The set $QI(X)$ of all equivalence classes of quasi-isometries of $X$ forms a group under composition: $[f].[g]=[f\circ g]$, for $[f],[g]\in QI(X)$.\\
	Any quasi-isometry $f:X\rightarrow X'$ induces an isomorphism $QI(X)\rightarrow QI(X')$ defined as $[h]\mapsto [f\circ h\circ g]$ where $g:X'\rightarrow X$ is a quasi-inverse of $f$. For example, $t\mapsto [t]$ is a quasi-isometry map from $\mathbb{R}$ to $\mathbb{Z}$. Therefore, $QI(\mathbb{R})$ is isomorphic to $QI(\mathbb{Z})$ as a group. The quasi-isometries of the real line can be represented by $PL$-homeomorphisms of the real line with bounded slopes.
	\subsection{$PL$-homeomorphisms of $\mathbb{R}$ with bounded slopes}
	Let $f:\mathbb{R}\rightarrow\mathbb{R}$ be any homeomorphism of $\mathbb{R}$. Denote by $B(f)$ the set of break points of $f$, that is points where $f$ fails to have derivative and by $\Lambda(f)$ the set of slopes of $f$, that is, $\Lambda(f)=\{f'(t):t\in\mathbb{R}\setminus B(f)\}$. Note that $B(f)\subset \mathbb{R}$ is discrete if $f$ is piecewise differentiable.

	We say that a subset $\Lambda$ of $\mathbb{R}^{*}$) (the set of non-zero real numbers), is \textit{bounded} if there exists an $M>1$ such that $M^{-1}<|\lambda|<M$ for all $\lambda\in \Lambda$. We say that a homeomorphism of $\mathbb{R}$ which is piecewise differentiable has bounded slopes if $\Lambda(f)$ is bounded.
	
    We denote by $PL_\delta(\mathbb{R})$ the set of all those piecewise-linear homeomorphisms $f$ of $\mathbb{R}$ such that $\Lambda(f)$ is bounded. It is clear that $PL_\delta(\mathbb{R})$ is a subgroup of the group $PL(\mathbb{R})$ of all piecewise linear homeomorphisms of $\mathbb{R}$. In \cite{Sankaran}, Sankaran established a connection between two groups $QI(\mathbb{R})$ and $PL_\delta(\mathbb{R})$, which is given as the following theorem.
\begin{Thm}\label{2.1}
	The natural homomorphism $\phi:PL_\delta(\mathbb{R})\rightarrow QI(\mathbb{R})$, defined as $f\mapsto [f]$, is surjective.
\end{Thm}
\subsection{Left-orderable and locally indicable group}  
	\begin{defn}
		\textit{A group $G$ is left-orderable if it admits a total order $\leq$ on $G$ such that $g\leq h$ implies $fg\leq fh$ for any $f\in G$.}
	\end{defn}
	Clearly, every subgroup of a left-orderable group is left-orderable. More interestingly, an arbitrary product $\Gamma$ of left-orderable groups $\Gamma_{\lambda}$ is left-orderable. All torsion-free abelian groups, free groups, braid groups, the group Homeo$_{+}(\mathbb{R})$ of orientation-preserving homeomorphisms of the line, the fundamental groups of orientable surfaces and $\mathcal{G}_\infty$ (the group of germs at $\infty$ of homeomorphisms of $\mathbb{R}$) are some examples of left-orderable groups \cite{Mann}.
	
	It can be observed from the definition that any finite group is non left-orderable. An significant result due to Morris-Witte \cite{Deroin Navas Rivas} establishes that finite index subgroups of SL$(n,\mathbb{Z})$ are non left-orderable for $n\geq 3$.

	If $\leq $ is a left order in a group $G$, then $f\in G$ is said to be positive (resp. negative) if $f>id$ (resp. $f<id$). From this left order $\leq$, we may define an order $\leq ^{*}$ by letting $f\leq ^{*}g$ whenever $f^{-1}>g^{-1}$. Then the order $\leq ^{*}$ turns out to be right-invariant. One can certainly go the other way around, producing left-orders from right-orders. Consequently, a group is left-orderable if and only if it is right-orderable. We refer the reader to \cite{Deroin Navas Rivas} for a detailed description related to the left-orderablility of a group.\\
	A link between orders and dynamics comes from the following theorem relating left-invariant orders to actions on the line.
	\begin{Thm} (see \cite{Mann} and Theorem 6.8 in \cite{Ghys})\label{2.2}
		Let $G$ be a countable group. Then $G$ is left-orderable if and only if there is an injective homomorphism $G\rightarrow$ Homeo$_{+}(\mathbb{R})$. Moreover, given an order on $G$, there is a canonical (up to conjugacy in Homeo$_{+}(\mathbb{R})$) injective homomorphism $G\rightarrow$ Homeo$_{+}(\mathbb{R})$ called a dynamical realization.
	\end{Thm}
In \cite{Navas}, Navas came up with a useful criterion for left-orderablility of a group, which is given below.
\begin{Lem}
	A group $G$ is left-orderable if and only if for every collection of nontrivial elements $g_1,g_2,...,g_k$, there exist choices $\epsilon_i\in\{1,-1\}$ such that the identity is not an element of the semigroup generated by $\{g_i^{\epsilon_i},i=1,2,...,k\}$.
\end{Lem}
It is clear that this condition is necessary, that is, if $G$ is left-orderable, then we can choose $\epsilon_i\in\{-1,1\}$ such that $g_i^{\epsilon_i}>id$ holds for each $i$. But much effort is needed to prove the condition is sufficient; we refer the reader to Prop. 1.4 of \cite{Navas} for a proof. 
	\begin{defn}
		\textit{A group is called locally indicable if each of its finitely generated nontrivial subgroups admits a nontrivial homomorphism to $\mathbb{Z}$.}
	\end{defn}
	\begin{ex}
		A remarkable theorem independently obtained by Brodskii and Howie \cite{Deroin Navas Rivas} asserts that torsion-free, 1-relator groups are locally indicable. Also, all knot groups in $\mathbb{R}^3$ are locally indicable. The Braid groups $\mathbb{B}_3$ and $\mathbb{B}_4$ are locally indicable but the groups $\mathbb{B}_n$ fail to be locally indicable for $n\geq 5$. See \cite{Deroin Navas Rivas} for the examples given above.
	\end{ex}
	\begin{defn}
		A left-order $\leq $ on a group $G$ is said to be Conradian (or a $C$-order) if for all positive elements $f,g$ there exists $n\in\mathbb{N}$ such that $fg^n>g$.
	\end{defn}
	The most important theorem regarding $C$-orderable group is the next one proved in \cite{Deroin Navas Rivas}.
	\begin{Thm}\label{2.3}
		A group $G$ is $C$-orderable if and only if it is locally indicable.
	\end{Thm}
	\section{Structure Theorem}
	It is shown that $QI(\mathbb{R_+})$ is not simple by providing a normal subgroup $H=\big\{[f]\in QI(\mathbb{R_+}): \displaystyle\lim_{x\rightarrow \infty}\frac {f(x)}{x}=1\big\}$ of $QI(\mathbb{R_{+}})$ in Theorem 2.2 of \cite{Ye Zhao}. Similarly, $QI(\mathbb{R}_{-})$ is not simple with a normal subgroup $H'=\big\{[f]\in QI(\mathbb{R}_{-}):\displaystyle\lim_{x\rightarrow \infty}\frac {f(x)}{x}=1\big\}$. Then, $\overline{H}=H\times H'$ is also a normal subgroup of $QI^{+}(\mathbb{R})$, where $QI^{+}(\mathbb{R})$ is the orientation preserving quasi-isometries of $\mathbb{R}$. It can be readily checked that $\overline{H}:=\big\{[f]\in QI^{+}(\mathbb{R}): \displaystyle\lim_{|x|\rightarrow \infty}\frac {f(x)}{x}=1\big\}$ is a normal subgroup in $QI(\mathbb{R})$.\\
	The construction of $H$ motivates us to define a set $S_{[g]}$ for any $[g]\in QI(\mathbb{R_{+}})$ as given in the Introduction. This is an invariant that we associate to an element $[g]$. Using this set as an invariant, we give a characterization of the elements of the quotient group $QI(\mathbb{R_{+}})/H$. From here we denote $QI(\mathbb{R_{+}})$ by $G$ for convenience of notation. The following result gives topological characterization of $S_{[g]}$.
	\begin{Lem}\label{3.1}
		\textit{The set $S_{[g]}$ is a compact and connected subset of $\mathbb{R_+}$.}
	\end{Lem}
	\begin{proof}
		Without loss of generality, we assume that $g(0)=0$. Now, since $[g]$ is a quasi-isometry, there exists $K>1$ such that $\frac {1}{K}x-K\leq g(x)-g(0)\leq Kx+K$. Therefore, as $x>0$, $\frac {1}{K}-\frac {K}{x}\leq \frac {g(x)}{x}\leq K+\frac {K}{x}$. So for each sequence $\{x_n\}$, $\frac {g(x_n)}{x_n}$ has a convergent subsequence. Thus, $S_{[g]}$ is a nonempty bounded subset of $\mathbb{R_+}$.\\
		To show $S_{[g]}$ is compact, it remains to show $S_{[g]}$ is closed. \\
		Let $\{\alpha_n\}$ be a sequence in $ S_{[g]}$ such that $\alpha_n\rightarrow \alpha$. Then for each $m\in\mathbb{N}$ there exist $\alpha_m$ and $\epsilon_m>0$ such that $(\alpha_m-\epsilon_m,\alpha_m+\epsilon_m)\subset (\alpha-\frac {1}{m},\alpha+\frac {1}{m})$. Since $\alpha_m\in S_{[g]}$, we can find a sequence $\{y_m\}$ divergent to $\infty$ and $\frac {g(y_m)}{y_m}\in (\alpha_m-\epsilon_m,\alpha_m+\epsilon_m)$. Therefore, we obtain a sequence $\{z_k\}$ such that $\frac {g(z_k)}{z_k}\rightarrow \alpha$. \\
		Now we show that $S_{[g]}$ is a connected set. Let $\alpha,\beta\in S_{[g]}$. Then there exist two sequences $\{x_n\},\{y_n\}$ such that $\displaystyle\lim_{n\rightarrow\infty}\frac {g(x_n)}{x_n}=\alpha$ and $\displaystyle\lim_{n\rightarrow\infty}\frac {g(y_n)}{y_n}=\beta$.\\
		Let $\gamma\in \mathbb{R}$ with $\alpha<\gamma<\beta$. Then we choose $k\in \mathbb{N}$ satisfying $\alpha+\frac {1}{k}\leq \gamma\leq \beta-\frac {1}{k}$. Now, there exist two sequences, $\{x_{r_k}\}$ and $\{y_{r_k}\}$ such that \\
		$\alpha-\frac {1}{k}\leq \frac {g(x_{r_k})}{x_{r_k}}\leq \alpha+\frac {1}{k}$ and $\beta-\frac {1}{k}\leq \frac {g(y_{r_k})}{y_{r_k}}\leq \beta+\frac {1}{k}$. By \cite[Theorem 1.2]{Sankaran}, we can choose $g$ to be piecewise linear homeomorphism, so $\frac {g(x)}{x}$ is continious for $x>0$.\\
		Then there exists a sequence $\{z_{r_k}\}$ such that $\displaystyle\lim_{k\rightarrow\infty}\frac {g(z_{r_k})}{z_{r_k}}=\gamma$. Hence $\gamma\in S_{[g]}$ and consequently $S_{[g]}$ is connected. This completes the proof.
	\end{proof}
	\begin{Rem}
		(i) From the above lemma, it is clear that $S_{[g]}$ is either a singleton set or a closed and bounded interval. Note that $H_c\subset H_I$ and $QI(\mathbb{R_{+}})=\displaystyle\bigcup_{I\subset (0,\infty)} H_I$.\\
		(ii) It can be readily checked that $[g]\in H_c$ if and only if $\displaystyle\lim_{x\rightarrow\infty}\frac {g(x)}{x}=c$.\\
		(iii) From the previous point, it follows easily that $\overline{H_c}$ is a subgroup of $G$.  
	\end{Rem}
	Let $\mathbb{I}$ be the collection of all compact sub-intervals of $(0,\infty)$. Note that there is a natural map $\psi$: $\mathbb{R}_{>0}\times\mathbb{I}\rightarrow \mathbb{I}$ defined by $(c',[a,b])\mapsto[c'a,c'b].$ It is easy to see that the previous lemma induces a natural map $\phi:QI(\mathbb{R_{+}})\rightarrow \mathbb{I}.$ Now, we characterize the following properties of the map $\phi$ and $\psi$.
	
	\begin{Prop}
		$(i)$ We have the following commutative diagram--
		\begin{center}
			$\begin{CD}
				\overline{H_c}\times QI(\mathbb{R_{+}}) @>(\phi,\phi)>> \mathbb{R}_{>0}\times \mathbb{I}\\
				@VV\circ V @VV\psi V\\
				QI(\mathbb{R_{+}}) @>\phi>> \mathbb{I}
			\end{CD}$
		\end{center} 
		where $\circ$ is the composition of quasi-isometry maps. \\
		$(ii)$ The map $\phi$ is surjective.
	\end{Prop}
	\begin{proof}
		$(i)$ Suppose $[f]\in \overline{H_c}$ and $[g]\in QI(\mathbb{R_{+}})$. Let $S_{[f]}=c',S_{[g]}=[a,b](\subset (0,\infty))$. \\
		Then there exists a sequence $\{x_n\}$ such that $x_n\rightarrow\infty $ and $\displaystyle\lim_{n\rightarrow\infty}\frac {g(x_n)}{x_n}=a$.\\ So, there exists a subsequence $\{x_{r_n}\}$ of $\{x_n\}$ such that \\ $\displaystyle\lim_{n\rightarrow\infty}\frac {f\circ g(x_{r_n})}{x_{r_n}}=\displaystyle\lim_{n\rightarrow\infty}\frac {f(g(x_{r_n}))}{g(x_{r_n})}~\frac {g(x_{r_n})}{x_{r_n}}=c'a$. So, $c'a\in\phi([f]\circ [g])$.
		\\Similarly, we can show that $c'b\in\phi([f]\circ [g])$. Then by using Lemma \ref{3.1}, we get, 
		\begin{equation}\label{p2 eq1}
			[c'a,c'b]\subset \phi([f]\circ [g]).
		\end{equation}
		To prove the other inclusion, assume  $\phi([f]\circ [g])=[a_1,b_1]$. Then by (\ref{p2 eq1}), we immediately get,
		\begin{equation}\label{p2 eq2}
			[c'^{-1}a_1,c'^{-1}b_1]\subset\phi([f]^{-1}\circ [f]\circ [g])=\phi[g].
		\end{equation}
		Then by (\ref{p2 eq1}) and (\ref{p2 eq2}), we have $\phi([f]\circ [g])=[c'a,c'b]$.\\
		$(ii)$ Let $[a,b]\in\mathbb{I}$. Now we want to construct an element $[g]\in QI(\mathbb{R_{+}})$ such that $S_{[g]}=[a,b]$.  For that, first, we choose $\lambda>max\{1,\frac {a}{2b-a},\frac {2b}{a}-1\}$ and define $g$ on $[0,1]$ as $g(x)=x$ and on $[\lambda^{n},\lambda^{n+1}]$ for all $n\geq 0$ as follows: 
		\begin{equation}
			g(x)=\begin{cases}
				g(\lambda^{n})+(\frac {\lambda+1}{\lambda-1}a-\frac {2b}{\lambda-1})(x-\lambda^n) , & \text{if $x\in [\lambda^n,\frac {\lambda^n+\lambda^{n+1}}{2}]$}\\
				g(\frac {\lambda^{n}+\lambda^{n+1}}{2})+(\frac {2\lambda b}{\lambda-1}-\frac {\lambda+1}{\lambda-1}a)(x-\frac {\lambda^n+\lambda^{n+1}}{2}), & \text{if $x\in  [\frac {\lambda^n+\lambda^{n+1}}{2},\lambda^{n+1}]$}.
			\end{cases}
		\end{equation}
		The function $g$ can be rewritten as follows:
		\begin{equation*}
			g(x)=\begin{cases}
				1+(\lambda^n-1)b+(\frac {\lambda+1}{\lambda-1}a-\frac {2b}{\lambda-1})(x-\lambda^n) , & \text{if $x\in [\lambda^n,\frac {\lambda^n+\lambda^{n+1}}{2}]$}\\
				1-b+\frac {a}{2}\lambda^{n}(\lambda+1)+(\frac {2\lambda b}{\lambda-1}-\frac {\lambda+1}{\lambda-1}a)(x-\frac {\lambda^n+\lambda^{n+1}}{2}), & \text{if $x\in  [\frac {\lambda^n+\lambda^{n+1}}{2},\lambda^{n+1}]$}.
			\end{cases}
		\end{equation*}
		Now we consider two sequences $\{x_n\}$ and $\{y_n\}$ such that ${x_n}=\lambda^{n}$ and ${y_n}=\frac {\lambda^n+\lambda^{n+1}}{2}$. Then it is clear from the definition of $g$ that $\frac{g(x_n)}{x_n}\rightarrow b$ and $\frac {g(y_n)}{y_n}\rightarrow a$. \\
		Since $\frac {\lambda+1}{\lambda-1}a-\frac {2b}{\lambda-1}<b,\frac {2\lambda b}{\lambda-1}-\frac {\lambda+1}{\lambda-1}a>a$, then from Lemma \ref{3.1}, we get $S_{[g]}=[a,b]$, that is, $\phi$ is surjective.
	\end{proof}
    Suppose $S_{[g]}=[a,b]$. Then we can choose an element $[f_{a^{-1}}]\in H_{a^{-1}}$ such that $S_{[f_{a^{-1}}\circ g]}=[1,a^{-1}b]$. Since $\overline{H_c}$ is a subgroup of $G$, so we can write the right coset of $G/\overline{H_c}$ in the following expression.
	\begin{align*}
		G/\overline{H_c}&=\bigsqcup_{M\geq 1}H_{[1,M]}/H\\
		&=\{*\}\bigsqcup\big\{\bigsqcup_{M> 1}H_{[1,M]}/H\}
	\end{align*}
	where $\{*\}$ stands for the trivial coset $H$ in $H_{[1,M]}/H$ and $H_{[1,M]}/H$ denotes the equivalence classes of $H_{[1,M]}$ where $H$ acts on $H_{[1,M]}$ as left multiplication.\\
	By Sankaran \cite{Sankaran}, we know that every equivalence class in $G=QI(\mathbb{R_{+}})$ can be represented by a piecewise linear homeomorphism with slopes, both bounded above and bounded away from 0. \\
	But it seems to be challenging to say when two piecewise-linear homeomorphisms lie in the same equivalence class. Instead, if we consider $H_{[1,M]}/H$, a  classification of equivalence classes in $H_{[1,M]}/H$ can be given in terms of the slopes of the representations from the piecewise-linear homeomorphisms.\\
	Suppose $[f],[g]\in H_{[1,M]}.$ If $H[f]=H[g]$ and if $\{a_n\}$ is a sequence such that $\{a_n\}\rightarrow\infty$ and $\displaystyle\lim_{n\rightarrow\infty}\frac {f(a_n)}{a_n}=1$, then $\displaystyle\lim_{n\rightarrow\infty}\frac {g(a_n)}{a_n}=1$. In other way, we can say that if no such sequence $\{a_n\}$ exists such that $\displaystyle\lim_{n\rightarrow\infty}\frac {f(a_n)}{a_n}=1$ and $\displaystyle\lim_{n\rightarrow\infty}\frac {g(a_n)}{a_n}=1$, then $H[f]\neq H[g]$. If there exists a sequence $\{a_n\}$ such that $\{a_n\}\rightarrow \infty$ and $\frac {f(a_n)}{a_n}\rightarrow 1,\frac {g(a_n)}{a_n}\rightarrow 1,$ then we provide some conditions on the slope of $f$ and $g$ to ensure the equality of two equivalence classes $H[f]$ and $H[g]$ and conversely. For this, we need certain notions, which is already defined in Introduction, before Theorem \ref{1.1}.\\

    To prove Theorem \ref{1.1}, we first need to observe that $||A_n-B_n||_{\infty}(=\{\max\limits_{1\leq i\leq  r}|\lambda_{n,i}-\lambda_{n,i}^{'}|\})\rightarrow 0$ if and only if for any $\epsilon>0$, there exists $N\in\mathbb{N}$ such that $|\lambda_i^n-\lambda_i^{'n}|<\epsilon$ for $n\geq N$. We also need the following result which gives an equivalent statement about the equality of two equivalence classes of $H_{[1,M]}/H$.
	\begin{Lem}{\label{3.5}}
		\textit{If $H[f],H[g]\in H_{[1,M]}/H$, then $f,g$ belong to the same equivalence class if and only if for every $\epsilon>0$, there exists an $M>0$ such that $\big|\frac {f(x)}{x}-\frac {g(x)}{x}\big|<\epsilon$, for all $x>M$.}
	\end{Lem}
	\begin{proof}
		First, we let $H[f],H[g]$ belong to the same equivalence class. Then, $f=h\circ g$, for some $h\in H$. Then 
		\begin{align*}
			\bigg|\frac {f(x)}{x}-\frac {g(x)}{x}\bigg|=\bigg|\frac {h(g(x))}{x}-\frac {g(x)}{x}\bigg|=\bigg|\frac {h(g(x))}{g(x)}\frac {g(x)}{x}-\frac {g(x)}{x}\bigg|=\bigg|\frac {g(x)}{x}\bigg|~\bigg|\frac {h(g(x))}{g(x)}-1\bigg|.
		\end{align*}
		Since $g$ is a quasi-isometry, $\frac {g(x)}{x}$ is bounded, that is, there exists $M_1>0$ such that $\big|\frac {g(x)}{x}\big|<M_1$. As $h\in H$, then for $\epsilon>0$, there exists $M_2\in\mathbb{R}_{>0}$ such that $\big|\frac {h(g(x))}{g(x)}-1\big|<\frac {\epsilon}{M_1}$ for all $x\geq M_2$. Let $M'=max\{M_1,M_2\}$. Then from above we get,
		$\big|\frac {f(x)}{x}-\frac {g(x)}{x}\big|<\epsilon$, for all $x>M'$.\\
		To prove the converse part, we first write, $f=h'\circ g$. Now we show that $h'\in H$. Now,
		$\big|\frac {f(x)}{x}-\frac {g(x)}{x}\big|<\epsilon$ implies $\big|\frac {g(x)}{x}\big|\big|\frac {h(g(x))}{g(x)}-1\big|<\epsilon$. \\
		Note that, $\big|\frac {g(x)}{x}\big|>\frac {1}{2}$, for all $x>M_1'$, for some $M_1'\in\mathbb{N}$. If not, suppose there exists a sequence $\{x_n\}$ such that $\{x_n\}\rightarrow\infty$ and $\frac {g(x_n)}{x_n}\leq \frac {1}{2},$ so it has a subsequence $\{x_{r_n}\}$ such that $\displaystyle\lim_{n\rightarrow\infty}\frac {g(x_{r_n})}{x_{r_n}}\leq \frac {1}{2}$, which is absurd as $[g]\in H_{[1,M]}$.\\
		Since $\big|\frac {g(x)}{x}\big|>\frac {1}{2}$ for all $x>M_1'$, then by putting $g(x)=y$, we get, $\big|\frac {h(y)}{y}-1\big|<2\epsilon$ for all $y>M_1''$. Hence $h\in H$.
	\end{proof}
	We are now ready to prove Theorem \ref{1.1}.
	\subsection{Proof of Theorem \ref{1.1}}
	Let $x\in[a_n,a_{n+1}]$. Then we get two points of the partition of $[a_n,a_{n+1}]$ say $x_{n,i-1}$ and $x_{n,i}$ as discussed in Introduction so that $x_{n,i-1}\leq x\leq x_{n,i}$ and we can write, \\
	$f(x)=f(a_n)+\lambda_{n,1}(x_{n,1}-a_n)+\lambda_{n,2}(x_{n,2}-x_{n,1})+\ldots+\lambda_{n,i}(x-x_{n,i-1})$ and\\ $g(x)=g(a_n)+\lambda_{n,1}^{'}(x_{n,1}-a_n)+\lambda_{n,2}^{
	'}(x_{n,2}-x_{n,1})+\ldots+\lambda_{n,i}^{'}(x-x_{n,i-1})$. Therefore,
	\begin{align*}
		&\frac {f(x)-g(x)}{x}\\&=\frac {f(a_n)-g(a_n)+(\lambda_{n,1}-\lambda_{n,1}^{'})(x_{n,1}-a_n)+\ldots+(\lambda_{n,i}-\lambda_{n,i}^{'})(x-x_{n,i-1})}{x}\\&=\frac {f(a_n)-g(a_n)}{a_n}\frac {a_n}{x}+\frac {(\lambda_{n,1}-\lambda_{n,1}^{'})(x_{n,1}-a_n)+\ldots+(\lambda_{n,i}-\lambda_{n,i}^{'})(x-x_{n,i-1})}{x}.
	\end{align*}
	$(a)$ Since $||A_n-B_n||_{\infty}\rightarrow 0$, from above we get,
	\begin{align*}
		&\Big|\frac {f(x)}{x}-\frac {g(x)}{x}\Big|\\&\leq\Big|\frac {f(a_n)-g(a_n)}{a_n}\Big|\Big|\frac {a_n}{x}\Big|+\frac {|\lambda_{n,1}-\lambda_{n,1}^{'}|(x_{n,1}-a_n)+\ldots+|\lambda_{n,i}-\lambda_{n,i}^{'}|(x-x_{n,i-1})}{|x|}\\&\leq \epsilon+\frac {\epsilon(x_{n,1}-a_n)+\ldots+\epsilon(x-x_{n,i-1})}{x}, ~\text{for}~n\geq N, ~\text{for~some~}n\in\mathbb{N}~\text{and}\\&\text{since }\frac {f(a_n)}{a_n},\frac {g(a_n)}{a_n}\rightarrow 1.~\text{Hence},
	\end{align*}
	\begin{align*}
		\Big|\frac {f(x)}{x}-\frac {g(x)}{x}\Big|\leq \epsilon+\epsilon\frac {x-a_n}{x}<2\epsilon.
	\end{align*}
	So, by Lemma \ref{3.5}, we have $H[f]=H[g]$.\\
	$(b)$ Since $H[f]=H[g]$, by Lemma \ref{3.5}, for every $\epsilon>0$, there exists $M>0$ such that $\big|\frac {f(x)}{x}-\frac {g(x)}{x}\big|<\epsilon$, for all $x>M$.\\
	After putting $x=x_{n,i}$ in the above inequality, we get $|A+B|<\epsilon$, where $A=$
	\begin{align*}
		\frac {f(a_n)-g(a_n)}{a_n}\frac {a_n}{x_{n,i}}+\frac {(\lambda_{n,1}-\lambda_{n,1}^{'})(x_{n,1}-a_n)+\ldots+(\lambda_{n,i-1}-\lambda_{n,i-1}^{'})(x_{n,i-1}-x_{n,i-2})}{x_{n,i}},
	\end{align*}
	\begin{align*}
		\text{and}~	B=\frac {(\lambda_{n,i}-\lambda_{n,i}^{'})(x_{n,i}-x_{n,i-1})}{x_{n,i}}.
	\end{align*}
	Now, the following four cases may arise:-\\
	Case I: Let $A,B>0$. Then $B<\epsilon$, so
	\begin{align*}
		\frac {(\lambda_{n,i}-\lambda_{n,i}^{'})(x_{n,i}-x_{n,i-1})}{x_{n,i}}<\epsilon.
	\end{align*}
	Since $x_{n,i}>K_1x_{n,i-1}$, we have $1-\frac {x_{n,i-1}}{x_{n,i}}>1-\frac {1}{K_1}.$ Thus from above we get,\\
	$\lambda_{n,i}-\lambda_{n,i}^{'}<\frac {\epsilon}{K_2}$, where $K_2=1-\frac {1}{K_1}$.\\
	Case II: Let $A,B<0$. Then $-B<\epsilon$, so
	\begin{align*}
		-\frac {(\lambda_{n,i}-\lambda_{n,i}^{'})(x_{n,i}-x_{n,i-1})}{x_{n,i}}<\epsilon.
	\end{align*}
	Then arguing as Case I, we get, $-(\lambda_{n,i}-\lambda_{n,i}^{'})<\frac {\epsilon}{K_2}$.\\
	Case III: Let $A<0,B>0$. Then, $B<\epsilon-A$. We have, $-A=$\\
	\begin{align*}
		-\frac {f(a_n)-g(a_n)+(\lambda_{n,1}-\lambda_{n,1}^{'})(x_{n,1}-a_n)+..+(\lambda_{n,i-1}-\lambda_{n,i-1}^{'})(x_{n,i-1}-x_{n,i-2})}{x_{n,i-1}}\frac {x_{n,i-1}}{x_{n,i}}.
	\end{align*} 
	Now,
	\begin{align*}
		&-\frac {f(a_n)-g(a_n)+(\lambda_{n,1}-\lambda_{n,1}^{'})(x_{n,1}-a_n)+..+(\lambda_{n,i-1}-\lambda_{n,i-1}^{'})(x_{n,i-1}-x_{n,i-2})}{x_{n,i-1}}\\&=-\frac {f(x_{n,i-1})-g(x_{n,i-1})}{x_{n,i-1}}<\epsilon~,\text{by~our~assumption}.
	\end{align*}
	So, $-A<\epsilon\frac {x_{n,i-1}}{x_{n,i}}<\frac {\epsilon}{K_1}$ and hence $B<\epsilon+\frac {\epsilon}{K_1}$. Thus, $\lambda_{n,i}-\lambda_{n,i}^{'}<\frac {\epsilon+\frac {\epsilon}{K_1}}{K_2}$.\\
	Case IV: Let $A>0,B<0$. Then arguing as Case III, one can show, $-B<\frac {\epsilon+\frac {\epsilon}{K_1}}{K_2}$.\\
	Combining four cases, $|\lambda_{n,i}-\lambda_{n,i}^{'}|<\epsilon_1$, where $\epsilon_1=max\Big\{\frac {\epsilon}{K_2},\frac {\epsilon+\frac {\epsilon}{K_1}}{K_2}\Big\}$.\\
	Hence, $||A_n-B_n||_{\infty}\rightarrow 0$ as $n\rightarrow\infty$.
	\begin{Rem}
		In Theorem \ref{1.1}, $||A_n-B_n||_{\infty}\rightarrow 0$ is a sufficient condition for $[f]$ and $[g]$ belonging to the same equivalence class of $H_{[1,M]}$, but not the necessary condition. In the following example, we show that the condition of Theorem \ref{1.1}, $(b)$ can not be relaxed.
	\end{Rem} 
	\begin{ex}
		Let $f,g:\mathbb{R^{+}}\rightarrow \mathbb{R^{+}} $ be defined by $f(x)=x$ and 
		\begin{align*}
			g(x)&=\frac {x}{2},x\in[0,1]\\&=g(2^n)+\frac {1}{2}(x-2^n), ~~x\in[2^n,2^{n}+1]~\\
			&=g(2^n)+\frac {1}{2}+\Big(\frac {1}{2}+\frac {1}{2^2}+...+\frac {1}{2^{n+2}}\Big)(x-2^n-1),~~x\in[2^n+1,2^{n+1}].
		\end{align*}
		First, we show that $f$ and  $g$ are not quasi-isometrically equivalent. It will be sufficient to prove that  $f(2^n+1)-g(2^n+1)>\frac {3}{4}+\frac {1}{2}+\frac {1}{2}+\ldots+\frac {1}{2}$ ($n$ times).
		One can easily check that $f(3)-g(3)>\frac {3}{4}+\frac {1}{2}.$\\
		Now, we assume that the result is true for $n-1$. Then,
		\begin{align*}
			&f(2^n+1)-g(2^n+1)\\&=2^n+1-g(2^{n-1}+1)-\Big(\frac {1}{2}+\frac {1}{2^2}\ldots+\frac {1}{2^{n+2}}\Big)(2^n-2^{n-1}-1)\\&=f(2^{n-1}+1)-g(2^{n-1}+1)+(2^n-2^{n-1})-\Big(\frac {1}{2}+\frac {1}{2^2}\ldots+\frac {1}{2^{n+2}}\Big)(2^n-2^{n-1}-1)\\&>\frac {3}{4}+\frac {1}{2}+\ldots+\frac {1}{2} ~((n-1)~\text{times})+(2^n-2^{n-1})\Big[1-\Big(\frac {1}{2}+\frac {1}{2^2}\ldots+\frac {1}{2^{n+2}}\Big)\Big]\\&+\Big(\frac {1}{2}+\ldots+\frac {1}{2^{n+2}}\Big),~\text{by~using~induction}\\&>\frac {3}{4}+\frac {1}{2}+\ldots+\frac {1}{2} ~(n~\text{times}).
		\end{align*}
		To show $[f]$ and $[g]$  belong to the same equivalence class of $H_{[1,M]}$, it is enough to prove that $\frac {g(x)}{x}\rightarrow 1$ as $x\rightarrow\infty$ due to Lemma \ref{3.5} and the fact that $\frac {f(x)}{x}\rightarrow 1$ as $x\rightarrow\infty$.\\
		Let $\epsilon_n=1-\frac {g(2^n)}{2^n}$. Now, we can write,
		\begin{align*}
			2^{n+1}-g(2^{n+1})=2^{n+1}-g(2^n)-\frac {1}{2}-\Big(\frac {1}{2}+\frac {1}{2^2}+\ldots+\frac {1}{2^{n+2}}\Big)(2^{n+1}-2^n-1).
		\end{align*}
	Thus,
		\begin{align*}
			\epsilon_{n+1}=1-\frac {g(2^{n+1})}{2^{n+1}}& =1-\frac {1-\epsilon_n}{2}-\frac {1}{2^{n+2}}-\frac {(2^n-\frac {2^n}{2^{n+2}}-1+\frac {1}{2^{n+2}})}{2^{n+1}}\\&=\frac {\epsilon_n}{2}+\frac {1}{2^{n+3}}-\frac {1}{2^{2n+3}}+\frac {1}{2^{n+2}}.
		\end{align*}
		Then from above, we have $\epsilon_n-\epsilon_{n+1}= \frac {\epsilon_n}{2}-\frac {1}{2^{n+3}}+\frac {1}{2^{2n+3}}-\frac {1}{2^{n+2}}$.\\
		If $\epsilon_n\geq\frac {1}{2^{n+2}}-\frac {1}{2^{2n+2}}+\frac {1}{2^{n+1}}$, then we get,
		$\epsilon_{n+1}\geq\frac {1}{2^{(n+1)+2}}-\frac {1}{2^{2(n+1)+2}}+\frac {1}{2^{(n+1)+1}}$ and hence we can write that $\epsilon_{n+m}\geq\frac {1}{2^{(n+m)+2}}+\frac {1}{2^{(n+m)+1}}-\frac {1}{2^{2(n+m)+2}}$.\\
		Suppose there exists $N\in\mathbb{N}$ such that 
		$\epsilon_n<\frac {1}{2^{n+2}}-\frac {1}{2^{2n+2}}+\frac {1}{2^{n+1}}$ for all $n$. Then $\epsilon_n\rightarrow 0$ as $n\rightarrow\infty$.  Otherwise, there is a natural number $n_1>N$ such that\\
		$\epsilon_{n_1}\geq\frac {1}{2^{n_1+2}}-\frac {1}{2^{2n_1+2}}+\frac {1}{2^{n_1+1}}$.\\ From the above observation, $\epsilon_{n_1+m}\geq\frac {1}{2^{(n_1+m)+2}}+\frac {1}{2^{(n_1+m)+1}}-\frac {1}{2^{2(n_1+m)+2}}$. \\
		So, $\epsilon_{n_1}\geq \epsilon_{n_1+1}\geq\epsilon_{n_1+2}\geq\ldots\ldots$. \\Hence $\{\epsilon_n\}_{n\geq n_1}$ is  monotonically decreasing and bounded sequence, so is convergent and the limit is $0$. Therefore, $\frac {g(2^n)}{2^n}\rightarrow 1$ as $n\rightarrow\infty$.\\
		Since there exists a natural number $N$ such that $0<1-\frac {g(2^n)}{2^n}<\epsilon$ for all $n\geq N$,\\ we have,
		\begin{align*}
			1-\frac {g(2^n+1)}{2^n+1}&<1+(\epsilon-1)\frac {2^n}{2^n+1}-\frac {1}{2^{n+1}+2}\\&=\frac {1}{2^n+1}+\epsilon\frac {2^n}{2^n+1}-\frac {1}{2^{n+1}+2}\\&<\epsilon+\epsilon=2\epsilon.
		\end{align*}
		Therefore, $\frac {g(2^n+1)}{2^n+1}\rightarrow 1$ as $n\rightarrow\infty$. \\
		Then it can be readily seen that $\frac {g(x)}{x}\rightarrow 1$ as $x\rightarrow\infty$. Also note that $\frac {2^n+1}{2^n}\ngtr K_1$ for some $K_1>1$, so the first hypothesis in Theorem \ref{1.1} $(b)$ as well as the conclusion are not satisfied.
	\end{ex}
	We will apply Theorem \ref{1.1} to achieve a result on when two specific classes of quasi-isometries of positive real line commute in the form of Theorem \ref{5.4}.\\ Apart from giving necessary and sufficient condition for two elements of $H_{[1,M]}/H$ being equal, below we show that $H_{[1,M]}/H$ contains $C_{[0,1]}^{[1,M]}$, where $C_{[0,1]}^{[1,M]}$ denotes the set of all continuously differentiable functions $f:[0,1]\rightarrow [1,M]$ with $f(0)=1,f(1)=M$ and $f'(x)\neq 0$ for all $x\in[0,1]$.
	\begin{Prop}\label{3.7}
		The set $C_{[0,1]}^{[1,M]}$ is embedded in $H_{[1,M]}/H$.
	\end{Prop}
	\begin{proof}
		Let $f\in C^{[1,M]}_{[0,1]}$. We define $h_1:[0,1]\rightarrow \mathbb{R}$ by $h_1(x)=(1+x)f(x)+\frac {f(x)-2M}{2M-1}$ and $h_2:[1,2M]\rightarrow \mathbb{R}$ be a linear map such that $h_2(1)=\frac {4M^2-3M}{2M-1}$ and $h_2(2M)=(2M)$. Define $h:[0,2M]\rightarrow\mathbb{R}$ by
		\begin{equation*}
			h(x)=
			\begin{cases}
				h_1(x), & x\in[0,1]\\
				h_2(x), & x\in[1,2M].
			\end{cases}
		\end{equation*}
		Next, we define $g_f:\mathbb{R_{+}}\rightarrow\mathbb{R_{+}}$ by
		\begin{equation*}
			g_f(x)=
			\begin{cases}
				h(x), & x\in[0,2M]\\
				(2M)^n+((2M)^n-(2M)^{n-1})h\big(\frac {x-(2M)^n}{(2M)^n-(2M)^{n-1}}\big), & x\in[(2M)^n,(2M)^{n+1}].
			\end{cases}
		\end{equation*}
		Since $f$ is onto, one can easily check that $g_f$ is a quasi-isometry. This can also be verified that $f\mapsto g_f$ is one-to-one.\\
		For $x\in[(2M)^n,(2M)^{n+1}]$, we can choose $z\in[0,2M]$ such that $x=((2M)^n-(2M)^{n-1})z+(2M)^n$. From the face that $\frac {g_f(z)}{z}=f(z)$ for $z\in[0,1]$, one can show that $g_f\in H_{[1,M]}$.\\
		If $[g_{f_1}]=[g_{f_2}]$ in $H_{[1,M]}/H$, then $g_{f_2}$ and $h\circ g_{f_1}$ for some $h\in H$, lie in the same equivalence class. Then for a fixed $z\in [0,1]$, we choose $x_n=((2M)^n-(2M)^{n-1})z+(2M)^n$, and $\displaystyle\lim_{n\rightarrow\infty}\frac {g_{f_2}(x_n)}{x_n}=\displaystyle\lim_{n\rightarrow\infty}\frac {h\circ g_{f_1}(x_n)}{x_n}$, this gives $f_1(z)=f_2(z)$.\\
		Therefore, we get an embedding of $C_{[0,1]}^{[1,M]}$ of $H_{[1,M]}/H$.
	\end{proof}
\newpage
\section{Left-orderability and Locally indicability}
Now, we prove Theorem \ref{1.2}.
\subsection{Proof of Theorem \ref{1.2}}
	Let $f_1,f_2,...,f_n\in G/H$ be any finitely many non-trivial elements. Choose a sequence $\{x_{1,k}^{'}\}$ such that $x_{1,k}^{'}\rightarrow\infty$ as $k\rightarrow\infty$ and $\displaystyle\lim_{k\rightarrow\infty}\frac {f_1(x_{1,k}^{'})}{x_{1,k}^{'}}\neq 1$. For each $i>1$, the set of all subsequential limits of $\frac {f_i(x_{1,k}^{'})}{x_{1,k}^{'}}$ is either $\{1\}$ or not. After passing to subsequences, we finally obtain a subsequence $\{x_{1,k}\}$ of $\{x_{1,k}^{'}\}$ such that either  $\displaystyle\lim_{k\rightarrow\infty}\frac {f_i(x_{1,k})}{x_{1,k}}\neq1$ or $=1$. We assign $\epsilon_i=1$ when $\displaystyle\lim_{k\rightarrow\infty}\frac {f_i(x_{1,k})}{x_{1,k}}>1$, $\epsilon_i=-1$ when $\displaystyle\lim_{k\rightarrow\infty}\frac {f_i(x_{1,k})}{x_{1,k}}<1$ and we define $S_1=\{f_i:\displaystyle\lim_{k\rightarrow\infty}\frac {f_i(x_{1,k})}{x_{1,k}}=1\}$. If $S_1$ is not empty, choose $f_{i,0}\in S_1$. We choose another sequence $\{x_{2,k}^{'}\}$ such that $\displaystyle\lim_{k\rightarrow\infty}\frac {f_{i,0}(x_{2,k}^{'})}{x_{2,k}^{'}}\neq 1$. Similarly, after passing to a subsequence, we have for each $f\in S_1$, either $\displaystyle\lim_{k\rightarrow\infty}\frac {f(x_{2,k})}{x_{2,k}}\neq 1$ or $=1$. We assign $\epsilon_i=1$ when $\displaystyle\lim_{k\rightarrow\infty}\frac {f(x_{2,k})}{x_{2,k}}>1$, $\epsilon_i=-1$ when $\displaystyle\lim_{k\rightarrow\infty}\frac {f(x_{2,k})}{x_{2,k}}<1$ and define $S_2=\{f: f\in S_1~\text{and}~\displaystyle\lim_{k\rightarrow\infty}\frac {f(x_{2,k})}{x_{2,k}}=1\}$. Note that this process will stop after finitely many steps and if $f_s\in S_i\setminus S_{i+1}$, then $\displaystyle\lim_{k\rightarrow\infty}\frac {f_s(x_{r,k})}{x_{r,k}}=1$ for $r=1,2,...,i$ and $\displaystyle\lim_{k\rightarrow\infty}\frac {f_s^{\epsilon_s}(x_{i+1,k})}{x_{i+1,k}}>1$. \\
	We prove that all words with the set of letters subset of $S_t$, but not subset of $S_{t+1}$, (for $t=0,1,...$) are non-identity and we do this by induction on the lengths of the words. Let $w=f_{i_1}^{\epsilon_{i_1}}f_{i_2}^{\epsilon_{i_2}}...f_{i_m}^{\epsilon_{i_m}}\in<f_1,f_2,...,f_m>$ be a non trivial reduced word. Assuming $S_0=\{f_1,f_2,...,f_n\}$ we can choose $t$ so that $\{f_{i_1},f_{i_2},...,f_{i_m}\}\subset S_t$, but $\{f_{i_1},f_{i_2},...,f_{i_m}\}\nsubseteq S_{t+1}$. For the convenience of notation, we assume $\{f_{i_1},f_{i_2},...,f_{i_m}\}=\{h_1,h_2,...,h_k,g_1,g_2,...,g_l\}$, where $\{h_1,h_2,...,h_k,g_1,g_2,...,g_l\}\cap S_{t+1}=\{g_1,g_2,...,g_l\}$.\\
	Suppose we can write $w=w'h$. If $w'$ is a word with the set of letters, subset of $S_t$, but not of $S_{t+1}$, then by induction hypothesis, there exists  a subsequence $\{x_{t+1,k}^{'}\}$ of $\{x_{t+1,k}\}$ such that $\displaystyle\lim_{k\rightarrow\infty}\frac {w'(x_{t+1,k}^{'})}{x_{t+1,k}^{'}}>1$. Then two cases arise:\\
	Case I: $w'$ has at least one $h_i$. Then by induction, $\displaystyle\lim_{k\rightarrow\infty}\frac {w'(x_{t+1,k}^{'})}{x_{t+1,k}^{'}}>1$.\\
	Since $h\notin S_{t+1}$, so we have $\displaystyle\lim_{k\rightarrow\infty}\frac {h(x_{t+1,k}^{'})}{x_{t+1,k}^{'}}>1$. So, there exists $N\in\mathbb{N}$ such that for all $k\geq N$, $h(x_{t+1,k}^{'})>k'x_{t+1,k}^{'}$ (for some $k'>1.$) $>x_{t+1,k}^{'}$. \\
	Since all the elements can be chosen as piecewise-linear homeomorphisms with slopes greater than zero \cite{Sankaran}, therefore, $w'h(x_{t+1,k}^{'})>w'x_{t+1,k}^{'}$, for all $k\geq N$. Now,\\
	$w'h(x_{t+1,k}^{'})-x_{t+1,k}^{'}=w'h(x_{t+1,k}^{'})-w'(x_{t+1,k}^{'})+w'(x_{t+1,k}^{'})-x_{t+1,k}^{'}>w'(x_{t+1,k}^{'})-x_{t+1,k}^{'}.$
	\begin{align*}
		\text{Thus,}~\frac {w(x_{t+1,k}^{'})}{x_{t+1,k}^{'}}-1>\frac {w'(x_{t+1,k}^{'})}{x_{t+1,k}^{'}}-1>l,~\text{for~some~}l>0.
	\end{align*}
	So, there exists a subsequence $\{x_{t+1,k}^{''}\}$ of $\{x_{t+1,k}\}$ such that $\displaystyle\lim_{k\rightarrow\infty}\frac {w(x_{t+1,k}^{''})}{x_{t+1,k}^{''}}>1.$\\
	Case II: Suppose $w'$ has $g_i^{,s}$ only. Let $\displaystyle\lim_{k\rightarrow\infty}\frac {h(x_{t+1,k})}{x_{t+1,k}}=k_1>1$, then for $\epsilon>0$, there exists $N_1\in\mathbb{N}$ such that $|h(x_{t+1,k})-k_1x_{t+1,k}|<\epsilon|x_{t+1,k}|$, for all $k\geq N_1$. This implies
	\begin{align*}
		||h(x_{t+1,k})-x_{t+1,k}|-(k_1-1)x_{t+1,k}|\leq |h(x_{t+1,k})-k_1x_{t+1,k}|<\epsilon x_{t+1,k},~\text{for~all~} k\geq N_1.
	\end{align*}
	Then from above, we get,\\
	$-|h(x_{t+1,k})-x_{t+1,k}|+(k_1-1){x_{t+1,k}}<\epsilon x_{t+1,k}$, for all $k\geq N_1$. So, \\
	$|h({x_{t+1,k}})-x_{t+1,k}|>(k_1-1-\epsilon)x_{t+1,k}$, for all $k\geq N_1.$\\ 
	Since $w'$ is a quasi-isometry, we have $M>1,C>0$ and 
	\begin{center}
		$\frac {1}{M}|h(x_{t+1,k})-x_{t+1,k}|-C<|w'h(x_{t+1,k})-w'(x_{t+1,k})|$. 
	\end{center}
	Then from above, $\frac {1}{M}(k_1-1-\epsilon)|x_{t+1,k}|-C<|w'h(x_{t+1,k})-w'(x_{t+1,k})|$.\\
	Since $h(x_{t+1,k})>x_{t+1,k}$, we again have, $w'h(x_{t+1,k})>w'(x_{t+1,k})$, so we can write,
	\begin{equation*}
		\frac {w'(h(x_{t+1,k}))-w'(x_{t+1,k})}{x_{t+1,k}}>\frac {1}{M}(k_1-1-\epsilon)-\frac {C}{x_{t+1,k}}.
	\end{equation*}
	On the other hand, if $g_1,g_2\in S_{t+1}$, then there exists $N_2\in\mathbb{N}$ such that\\ $|g_i(x_{t+1,k})-x_{t+1,k}|<\epsilon x_{t+1,k},$ for all $k\geq N_2$ and $~i=1,2$. Then 
	\begin{align*}
		&|g_2\circ g_1(x_{t+1,k})-x_{t+1,k}|\leq M\epsilon{x_{t+1,k}}+C+\epsilon x_{t+1,k}\\&\implies
		\Big|\frac {g_2\circ g_1(x_{t+1,k})}{x_{t+1,k}}-1\Big|\leq M\epsilon+\frac {C}{x_{t+1,k}}+\epsilon\\&\implies \lim_{k\rightarrow\infty}\frac {g_2\circ g_1(x_{t+1,k})}{x_{t+1,k}}=1.
	\end{align*}
	Generalizing this, we get, $\displaystyle\lim_{k\rightarrow\infty}\frac {w'(x_{t+1,k})}{x_{t+1,k}}=1.$ Thus,
	\begin{align*}
		&\frac {w(x_{t+1,k})}{x_{t+1,k}}-1\\&=\frac {w'h(x_{t+1,k})-w'(x_{t+1,k})}{x_{t+1,k}}+\frac {w'(x_{t+1,k})}{x_{t+1,k}}-1\\&>\frac {1}{M}(k_1-1-\epsilon)-\frac {C}{|x_{t+1,k}|}-\epsilon>m>0.
	\end{align*}
	Thus, there exists a subsequence $\{x_{t+1,k}^{'''}\}$ of $\{x_{t+1,k}\}$ such that $\displaystyle\lim_{k\rightarrow\infty}\frac {w(x_{t+1,k}^{'''})}{x_{t+1}^{'''}}>1$.\\
	
	For the remaining case, suppose we write $w=w'g$. Then $w'$ has at least one $h_i$. So, $\displaystyle\lim_{k\rightarrow\infty}\frac {w'(x_{t+1,k}^{'})}{x_{t+1,k}^{'}}=k_2>1$, for some subsequence $\{x_{t+1,k}^{'}\}$ of $\{x_{t+1,k}\}$.\\
	Since $\displaystyle\lim_{k\rightarrow\infty}\frac {g(x_{t+1,k}^{'})}{x_{t+1,k}^{'}}=1,$ then there exists $N_3\in\mathbb{N}$ such that $|g(x_{t+1,k}^{'}-x_{t+1,k}^{'})|<\epsilon x_{t+1,k}^{'}$, for all $k\geq N_3$.
	Again since $w'$ is a quasi-isometry, therefore,\\
	$|w'g(x_{t+1,k}^{'})-w'(x_{t+1,k}^{'})|<M\epsilon x_{t+1,k}^{'}+C$, for some fixed $M>1$, $C>0$.
	\begin{align*}
		\text{Finally,~}	\frac {w(x_{t+1,k}^{'})}{x_{t+1,k}^{'}}-1&=\frac {w'g(x_{t+1,k}^{'})-w'(x_{t+1,k}^{'})}{x_{t+1,k}^{'}}+\frac {w'(x_{t+1,k}^{'})}{x_{t+1,k}^{'}}-1\\&>-M\epsilon-\frac {C}{x_{t+1,k}^{'}}+k_2-1-\epsilon^{'}.
	\end{align*}
	Then by suitably choosing $\epsilon$ and $\epsilon'$, we can show $\frac {w(x_{t+1,k}^{'})}{x_{t+1,k}^{'}}-1>k_3$, for some $k_3>0$. \\
	So, there exists a subsequence $\{x_{t+1,k}^{''}\}$ of $\{x_{t+1,k}\}$ such that $\displaystyle\lim_{k\rightarrow\infty}\frac {w(x_{t+1,k}^{''})}{x_{t+1,k}^{''}}>1$.\\
	This completes the induction step and hence the proof.\\
	
Since a left-orderable group is torsion-free, we get an obvious corollary of the Theorem $\ref{1.2}$.
\begin{cor}
	The quotient group $G/H$ is torsion-free.
\end{cor}
\begin{Rem}
	(i) Since $QI^{+}(\mathbb{R})/\overline{H}\cong QI(\mathbb{R}_{+})/H\times QI(\mathbb{R}_{-})/H'$ and a direct product of left-orderable groups is left-orderable, so $QI^{+}(\mathbb{R})/\overline{H}$ is left-orderable.\\
	(ii) Let $f(x)=-x,x\in\mathbb{R}$. Then $f$ has order two in $QI(\mathbb{R})/\overline{H}$, so $QI(\mathbb{R})/\overline{H}$ is not left-orderable, hence is not locally indicable.
\end{Rem}
Now, we prove Theorem \ref{1.3}.
\subsection{Proof of Theorem 1.3}
	We show that there exist $x,y,z\in G/H$ such that there does not exist any onto homomorphism from $\Gamma=<x,y,z:x^2=y^3=z^7=xyz>$ to $(\mathbb{R},+)$.\\
	Consider now the following two one-parameter families of elements of $PSL(2,\mathbb{R})$, where they are parametrized by a formal real exponent of $s$: 
	\begin{center}
		$a^s=\begin{pmatrix}
			1 & -s \\
			0 & 1 
		\end{pmatrix}$, $b^s=\begin{pmatrix}
			1 & 0 \\
			s & 1
		\end{pmatrix}$ and we define $d=aba=\begin{pmatrix}
			0 & -1 \\
			1 & 0
		\end{pmatrix}$.
	\end{center}
	Since an element $f$ of $PSL(2,\mathbb{R})$ sends a line through the origin to a line through the origin, it gives a diffeomorphism $f_1$ of $\mathbb{R}\mathbb{P}^{1}$, hence a diffeomorphism $f_2$ of $\mathbb{S}^1$ by the following commutative diagram.
	\newpage
	\begin{center}
\includegraphics[scale=1.4]{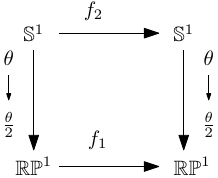}
	\end{center}
    where $\mathbb{R}\mathbb{P}^{1}=$upper half circle$/(1,0)\sim(-1,0)$.\\
	This diffeomorphism of $\mathbb{S}^1$ can be lifted to a unique homeomorphism $\widetilde{f}$ of $\mathbb{R}$ such that $\widetilde{f}(0)\in[0,1)$ due to the following diagram.
	\begin{center}
		\includegraphics[scale=1.4]{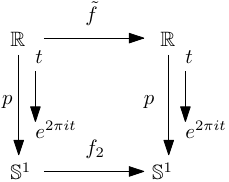}
	\end{center}
	Note that these liftings $\widetilde{f}\subset Diff_{\mathbb{Z}}(\mathbb{R})$, hence by Corollary 2.4 of \cite{Ye Zhao}, these liftings $\widetilde{f}$ can be embedded in $QI(\mathbb{R})$. We will find a subgroup $\Gamma=<x,y,z:x^2=y^3=z^7=xyz>$ of $Diff_{\mathbb{Z}}(\mathbb{R})$.\\
	For a real numbers $s$, let us define $A^s,B^s$ and $D$ to be the unique lifting of $a^s,b^s,d$ respectively, that is,  $A^r=\widetilde{a^s}$, $B^s=\widetilde{b^s}$ and $D=\tilde{d}$.\\ Consider $x=ABA$, $y=A^rBA^{1-r}$, $z=y^{-1}x$, where $r$ is a root of $r^2-r+2=2\cos(\frac {\pi}{7})$.\\
	It can be checked from the following description of $x,y$ and $z$; that each of them sends positive real axis to positive real axis, and so $\Gamma$ can be further embedded into $QI(\mathbb{R_{+}})$.
	To show $QI(\mathbb{R_+})/H$ is not locally indicable, it is enough to prove that there is no non-trivial homomorphism from $\Gamma$ to $(\mathbb{R},+)$ and the embeddings of $x,y,z$ are not in $H$.\\
	From the expression of $aba$ it is clear that,
	\begin{align*}
		(aba)_1:&~\mathbb{R}\mathbb{P}^{1}\rightarrow \mathbb{R}\mathbb{P}^{1}\\& \theta\mapsto\theta+\frac {\pi}{2},~~0\leq\theta\leq\frac {\pi}{2}\\&\theta\mapsto\theta-\frac {\pi}{2},~~\frac {\pi}{2}<\theta\leq \pi
	\end{align*} 
	and so,
	\begin{align*}
		(aba)_2:&~\mathbb{S}^1\rightarrow\mathbb{S}^1\\&\theta\mapsto\theta+\pi,~~0\leq \theta\leq \pi\\&\theta\mapsto\theta-\pi,~~\pi<\theta\leq 2\pi.
	\end{align*}
	Therefore, the map $\widetilde{aba}:\mathbb{R}\rightarrow\mathbb{R}$ is defined by $t\mapsto t+\frac {1}{2}$ and hence $\widetilde{aba}(0)=\frac {1}{2}.$\\
	Note that, $p\circ ABA=aba\circ p$ and hence $ABA-\widetilde{aba}=n$, for some fixed $n\in\mathbb{Z}$.\\
	Now we find out $ABA(0)$. Observe that $A(0)=0,B(0)=\frac {1}{4}$ and $A(\frac {1}{4})=\frac {1}{2}$. So, $ABA(0)=\frac {1}{2}$ and thus $ABA=\widetilde{aba}$.\\
	So, $x^2:\mathbb{R}\rightarrow\mathbb{R}$ is defined by $t\mapsto t+1$. Therefore, $x^2$ commutes with all $g:\mathbb{R}\rightarrow\mathbb{R}$ such that $g(t+1)=g(t)+1$.\\
	Similarly, one can show that $BAB=x$. Since $y$ is conjugate to $AB$, $(AB)^3=(ABA)(BAB)=x^2$ and $x^2$ commutes with $(A^{r-1})^{-1}$, so $y^3=x^2$. Also, $xyz=x^2$.\\
	Now, we take the map $\bar{z}:\mathbb{R}\mathbb{P}^{1}\rightarrow\mathbb{R}\mathbb{P}^{1}$, induced by $z:\mathbb{R}\rightarrow\mathbb{R}$.
	\begin{equation}\label{1.4}
		\text{So,~}\bar{z}=(a^rba^{1-r})^{-1}aba=\begin{pmatrix}
			r^2-r+1 & -r \\
			1-r & 1
		\end{pmatrix}=\begin{pmatrix}
			2\cos\frac {\pi}{7}-1 & -r \\
			1-r & 1
		\end{pmatrix},
	\end{equation} 
	since $r$ is a root of $r^2-r+2=2\cos(\frac {\pi}{7})$.\\
	By Bergman \cite{Bergman}, $\bar{z}$ has eigenvalues $e^{\pm\frac {\pi i}{7}}$, hence $\bar{z}^7=-I$. Then,
	\begin{equation}\label{1.5}
		{z}^7-id=n, \text{~for~some~fixed~}n\in\mathbb{Z}.
	\end{equation}
	From (\ref{1.4}) and for any values of $r$, we can see that both $(1,0)$ and $(0,1)$ in $\mathbb{R}\mathbb{P}^{1}$ rotates counterclockwise with an angle less than $\frac {\pi}{4}$ under $\bar{z}$, so does all the lines under $\bar{z}$. Hence $z$ will take $t$ to $t+\epsilon(t)$, where $0<\epsilon(t)<\frac {1}{4}$. So, $0<z^7(0)<\frac {7}{4}$, thus from (\ref{1.5}), $z^7(0)=1$, and hence $z^7:t\mapsto t+1$. Therefore, $z^7=x^2$, as desired.\\
	Now, we show the embeddings of $x,y $ and $z$ are not in H.\\
	Recall that from Corollary 2.4 of \cite{Ye Zhao},
	\begin{align*}
		h(x)&=e^x,~x\geq 1\\
		&=ex,~0\leq x\leq 1.
	\end{align*}
	Then $hxh^{-1}(t)=te^{\frac {1}{2}}$, so, $\frac {hxh^{-1}(t)}{t}=e^{\frac {1}{2}}\neq 1$, for all $t>e$. This implies the embedding of $x\notin H$.
	Similarly, the embedding of $y,z\notin H$.\\
	It is easy to check that there is no nontrivial homomorphism from $\Gamma$ to $(\mathbb{R},+).$
	This completes the proof.
	\section{Application}
	In the following, we provide a sufficient condition for a quasi-isometry of the positive real line to be in $\overline{H_c}$. This is also an attempt to describe the centralizer of some collections of quasi-isometries.
    \begin{Prop}\label{1.6}
	Let $f\in QI(\mathbb{R_{+}})$ such that\\
	there is a dense set $D$ of $(\lambda-\delta,\lambda+\delta)$ for some $\lambda>\delta>0$ and $|sf(x)-fs(x)|<M$, for all $s\in D$, $x\in \mathbb{R}.\hspace{2mm}-------$\hspace{8.9cm}(A)\\
	Then $f\in \overline{H_c}$.
	\end{Prop}
	\begin{proof}
	Let $\{a_n\}$ be a divergent sequence such that $\frac {f(a_n)}{a_n}$ converges. Then we can choose a subsequence $\{x_n\}$ of $\{a_n\}$ so that $\frac {f^{-1}(x_{n})}{x_{n}}\rightarrow d$. Let $\{y_n\}$ be any sequence.\\
	We assume $1>\lambda+\delta$ as other cases can be dealt similarly.\\ Then there exists $N\in\mathbb{N}$ such that $\displaystyle\bigcup_{m\geq N}((\lambda-\delta)^m,(\lambda+\delta)^m)=(0,(\lambda-\delta)^N)$. So, $\displaystyle\bigcup_{m\geq N}D^m$ is dense in $(0,(\lambda+\mu)^{N})$.\\
	We can choose a subsequence $\{x_{r_n}\}$ such that $\frac {f^{-1}(x_{r_n})}{y_n}>n$, so $\frac {y_n}{f^{-1}(x_{r_n})}<\frac {1}{n}$.\\
	Now, we can take $s_n\in\displaystyle\bigcup_{m\geq N}D^m$ such that $\Big|s_n-\frac {y_n}{f^{-1}(x_{r_n})}\Big|<\frac {1}{n}~\frac {y_n}{f^{-1}(x_{r_n})}$, equivalently, $\Big|s_n\frac {f^{-}(x_{r_n})}{y_n}-1\Big|<\frac {1}{n}$.
	For all $s^m\in\displaystyle\bigcup_{m\geq N}D^m,$ one can show, $|s^mf(x)-fs^m(x)|<(1+s+s^2+...+s^{m-1})M<\frac {M}{1-s}$. Thus, 
	\begin{align}
	|s_nf(x)-fs_n(x)|<\frac {M}{1-(\lambda+\delta)}.
	\end{align}
	Now,
	\begin{align*}
	\Big|\frac {f(y_n)}{y_n}-d\Big|&\leq \Big|\frac{f(s_nf^{-1}(x_{r_n}))}{s_nf^{-1}(x_{r_n})}-\frac {f(y_n)}{y_n}\Big|+\Big|\frac {f(s_nf^{-1}(x_{r_n}))}{s_nf^{-1}(x_{r_n})}-d\Big|\\&<\Big(2k+\frac {c}{\big|s_nf^{-1}(x_{r_n})\big|}\Big)\Big|1-\frac {s_nf^{-1}(x_{r_n})}{y_n}\Big|+\frac {c}{|y_n|}+\Big|\frac {f(s_nf^{-1}(x_{r_n}))}{s_nf^{-1}(x_{r_n})}-d\Big|\\&<\Big(2k+\frac {c}{\big|s_nf^{-1}(x_{r_n})\big|}\Big)\Big|1-\frac {s_nf^{-1}(x_{r_n})}{y_n}\Big|+\frac {c}{|y_n|}+\delta_n\\&<\Big(2k+\frac {c}{y_n(1-\frac {1}{n})}\Big)\frac {1}{n}+\frac {c}{|y_n|}+\delta_n\\&\leq\Big(2k+\frac {2c}{y_n}\Big)\frac {1}{n}+\frac {c}{|y_n|}+\delta_n
	\end{align*}
	where $c>0$ is a constant and we get $\delta_n$ by using \ref{1.6} where $\delta_n\rightarrow 0$.\\
	This completes the proof.
	\end{proof}
	A simple collection of examples which satisfy (A) of Prop. \ref{1.6} is $f_c(x)=cx,$ for $c>0,x\geq 0.$ Now, we give an example $g_\lambda\in QI(\mathbb{R_{+}})$ such that $g_\lambda$ commutes with $f_\lambda~(x\mapsto \lambda x)$ and $g_\lambda\notin \overline{H_c}$, for any $\lambda>0$.
	\begin{Ex}\label{1.7}
	We construct $g_\lambda$ for two cases:\\
	Case I: Let $\lambda>1$. Define $g_\lambda: \mathbb{R_{+}}\rightarrow\mathbb{R_{+}}$ by
	\begin{align*}
	g_\lambda(x)&=\frac {5}{2}x,~x\in [0,1]\\&=1+\frac {3}{2}\lambda^n+(x-\lambda^n),~x\in\Big[\lambda^n,\frac {\lambda^n+\lambda^{n+1}}{2}\Big]\\ &= 1+\lambda^n+\frac {\lambda^{n+1}}{2}+2\Big(x-\frac {\lambda^n+\lambda^{n+1}}{2}\Big),~x\in \Big[\frac {\lambda^n+\lambda^{n+1}}{2},\lambda^{n+1}\Big].
	\end{align*}
	It can be easily checked that $[g_\lambda]\in QI(\mathbb{R_+})$ and $[g_\lambda f_\lambda g_\lambda^{-1}]=[f_\lambda]$.\\
	Case II: Let $\lambda<1$. Then $\frac {1}{\lambda}>1$ and similarly as case I, $g_{\frac {1}{\lambda}}\in QI(\mathbb{R_{+}})$. Then $f_\lambda^{-1}(x)=f_{\frac {1}{\lambda}}(x)$ and $g_{\frac {1}{\lambda}}f_\lambda^{-1}=f_\lambda^{-1}g_{\frac {1}{\lambda}}$, which implies $[g_{\frac {1}{\lambda}}f_\lambda g_{\frac {1}{\lambda}}^{-1}]=[f_\lambda]$.\\
	
	Considering $\{\lambda^n\}$ and $\{\frac {\lambda^n+\lambda^{n+1}}{2}\}$ in Case I, we deduce $[g_\lambda]\notin \overline{H_c}$. Similarly $[g_{\frac 
	{1}{\lambda}}]$ (in Case II) $\notin \overline{H_c}$.\\
					
	An interesting fact is that $g_\lambda$ commutes with $f_c~(x\mapsto cx)$, for $c>0$ if and only if $c=\lambda^r$, for $r\in\mathbb{Z}$ and hence both $g_\lambda$ and $g_{\frac 
	{1}{\lambda}}$ do not satisfy the condition (A).
	\end{Ex}
	\begin{Rem}
	(i) The condition (A) in Prop. \ref{1.6} is more analytical than algebraic in nature. So, it can be a good question to ask whether the conclusion of Prop. \ref{1.6} can be deduced in spite of relaxing (A) so that $f$ only commutes with all the functions $f_c$ for $c\in D$.\\
	(ii) Apart from Example \ref{1.7}, one can search for an example $g$ such that $g\notin \overline{H_c}$ and $g$ commutes with $f_c$, $c\in D$ for some $D$ but does not satisfy (A).	
	\end{Rem}
As an application of Theorem \ref{1.1}, now we discuss when two large classes of quasi-isometries of positive real line commute. These classes are described in the following.
\subsection{Discussion}
Let $[f]\in H_{[1,M_1]},[g]\in H_{[1,M_2]}$ for $M_1,M_2>1$. Let $\{a_n\}$ be a strictly increasing sequence with $f(a_n)=g(a_n)=a_n$ and $x_n$ be the only break point for both $f$ and $g$ in $[a_n,a_{n+1}]$. \\
If $S_{[g\circ f]}\neq S_{[f\circ g]}$, then for any $h,h'\in H$, we can easily deduce that $S_{[h'\circ g\circ h\circ g]}\neq S_{[h\circ f\circ h'\circ g]}$, so $h'\circ g$ and $h\circ f$ do not commute. Nest we assume $S_{[g\circ f]}=S_{[f\circ g]}$.\\
				Now, we write the expression of $f$ and $g$ on $[a_n,a_{n+1}]$ in terms of the slopes $\lambda_n,\mu_n,\lambda_n^{'},\mu_n^{'}$ respectively as follows:
				\begin{align*}
					f(x)&=a_n+\lambda_n(x-a_n),~x\in [a_n,x_n]\\
					&=a_n+\lambda_n(x_n-a_n)+\mu_n(x-x_n),~x\in [x_n,a_{n+1}]\\\text{and}\\
					g(x)&=a_n+\lambda_n^{'}(x-a_n),~x\in [a_n,x_n]\\
					&=a_n+\lambda_n^{'}(x_n-a_n)+\mu_n^{'}(x-x_n),~x\in [x_n,a_{n+1}].
				\end{align*}
			Then $g\circ f$ have break-points $x_n$ and $f^{-1}(x_n)$ in $[a_n,a_{n+1}]$. Three cases may arise:\\
			Case (i): Let $x_n<f^{-1}(x_n)$.\\
			(a) If $x\in[a_n,x_n]$, then $g\circ f(x)=g(a_n+\lambda_n(x-a_n))=a_n+\lambda_n\lambda_n^{'}(x-a_n)$.\\
			(b) If $x\in[x_n,f^{-1}(x_n)]$, then $g\circ f(x)=a_n+\lambda_n^{'}(f(x)-a_n)=a_n+\lambda_n^{'}\lambda_n(x_n-a_n)+\lambda_n^{'}\mu_n(x-x_n)$.\\
			(c) If $x\in[f^{-1}(x_n),a_{n+1}]$, then $g\circ f(x)=a_n+\lambda_n^{'}(x_n-a_n)+\mu_n^{'}(f(x)-x_n)=a_n+\lambda_n^{'}(x_n-a_n)+\mu_n^{'}[a_n+\lambda_n(x_n-a_n)+\mu_n(x-x_n)-x_n]$.\\
			Thus from above, in this case the slopes of $g\circ f$ are $\lambda_n\lambda_n^{'}$, $\lambda_n^{'}\mu_n$, $\mu_n^{'}\mu_n$ respectively.\\
			Case (ii): Let $x_n=f^{-1}(x_n)$. Then the slopes of $g\circ f$ are $\lambda_n\lambda_n^{'}$ and $\mu_n^{'}\mu_n$ respectively.\\
			Case (iii): Let $x_n>f^{-1}(x_n)$. Then similarly as above, we can show that the slopes of $g\circ f$ are $\lambda_n\lambda_n^{'},\mu_n^{'}\lambda_n,\mu_n^{'}\mu_n$ respectively.

			For $f\circ g$, $x_n$ and $g^{-1}(x_n)$ are two break-points in $[a_n,a_{n+1}]$. Here also three cases arise:\\
			Case (i): If $x_n<g^{-1}(x_n)$, then the slopes of $f\circ g$ are $\lambda_n^{'}\lambda_n,\mu_n^{'}\lambda_n$ and $\mu_n^{'}\mu_n$ respectively.\\
			Case (ii): If $x_n=g^{-1}(x_n)$, then the slopes of $f\circ g$ are $\lambda_n^{'}\lambda_n$ and $\mu_n^{'}\mu_n$ respectively.\\
			Case (iii): If $x_n>g^{-1}(x_n)$, then the slopes of $f\circ g$ are $\lambda_n^{'}\lambda_n,\lambda_n^{'}\mu_n$ and $\mu_n^{'}\mu_n$ respectively.

				If $\lambda_n<1-\epsilon_1,~\lambda_n^{'}>1+\epsilon_2$ for some $0<\epsilon_1,\epsilon_2<1,$ then $g\circ f$ have break-points $x_n$ and $f^{-1}(x_n)$ in $[a_n,a_{n+1}]$ and $f\circ g$ have break-points $g^{-1}(x_n)$ and $x_n$. So, we get a partition $a_n<g^{-1}(x_n)<x_n<f^{-1}(x_n)<a_{n+1}$ of $[a_n,a_{n+1}]$ such that on each subintervals, both of $g\circ f$ and $f\circ g$ are linear. The difference of slopes of $g\circ f$ and $f\circ g$ in these subintervals are $0,|\lambda_n^{'}||\lambda_n-\mu_n|, |\mu_n||\lambda_n^{'}-\mu_n^{'}|$ and $0$ respectively.\\
				One can easily show that there exists $K_2>1$ with $\frac {g^{-1}(x_n)}{a_n},\frac {x_n}{g^{-1}(x_n)}>K_2$, for all $n\in\mathbb{N}$ and also compute $f^{-1}(x_n)=x_n+\frac {(1-\lambda_n)(x_n-a_n)}{\mu_n}$. Therefore,\\
				 $\frac {f^{-1}(x_n)}{x_n}=1+\frac {1-\lambda_n}{\mu_n}(1-\frac {a_n}{x_n})>1+\frac {\epsilon}{M}(1-\frac {1}{K})$, for some $M>1$. On the other hand, if for some subsequence $\{r_n\}$ of $\{n\}$, $f^{-1}(x_{r_n})\rightarrow a_{r_n+1}$, then $x_{r_n}\rightarrow f(a_{r_n+1})=a_{r_n+1}$ (since $f$ is chosen to be bilipschitz), which contradicts our assumption. So, there exists a $K_1>1$ such that $\frac {a_{n+1}}{f^{-1}(x_n)}>K_1>1$.\\
				 Thus the hypothesis of the Theorem \ref{1.1} (b) is satisfied, so $[g\circ f]=[f\circ g]~($mod $H)$ implies $|\lambda_n-\mu_n|\rightarrow 0$ and $|\lambda_n^{'}-\mu_n^{'}|\rightarrow 0$, which contradicts our previous assumption on $\lambda_n$ and $\lambda_n^{'}.$ So, $[g\circ f]\neq [f\circ g]~($mod $H)$. \\
				 For the remaining conditions on $\lambda_n$ and $\lambda_n^{'}$ as given in the following theorem, one can get the same conclusion. Hence we have proven the following result on commutativity of $f$ and $g$.
				 \begin{Thm}\label{5.4}
				 \textit{Let $[f],[g]\in H_{[1,\infty]}\big(=\displaystyle\bigcup_{M>1} H_{[1,M]}\big)$ with a strictly increasing sequence $\{a_n\}$ such that $f(a_n)=g(a_n)=a_n$ and $x_n$ is the only break-point in $[a_n,a_{n+1}]$. Let $\epsilon>0$ and $K,K'>1$ be constants such that $|\lambda_n-1|,|\lambda_n^{'}-1|,|\lambda_n-\lambda_n^{'}|>\epsilon$ and $K<\frac {x_n}{a_n},\frac {a_{n+1}}{x_n}<K'$. Then $h\circ f$ and $h'\circ g$ do not commute for any $h,h'\in H$.}
				 \end{Thm}
			 \begin{Rem}
			 	(i) The above condition, that is, there exists $\epsilon>0$ such that $|\lambda_n-1|,|\lambda_n^{'}-1|,|\lambda_n-\lambda_n^{'}|>\epsilon$ may make the impression that the hypothesis of the previous theorem is rather restrictive. But there is a large class of functions which satisfies this with the other conditions.\\
			 	(ii) We can construct $[g]\in H_{[1,M_2]}$ for some $M_2>1$ such that $\lambda_n^{'}\rightarrow 1$ and $\lambda_n^{'}\rightarrow\lambda_n$ in the disjoint sequence of intervals of the form $[a_n,a_{n+1}]$. Then $[g]$ is neither $[id]$ nor $[f]$, but $[g]$ commutes with $[f]$ (mod $H$).
			 \end{Rem}
\section*{Acknowledgement}
The authors thank Parameswaran Sankaran for his valuable suggestions and comments.

			\end{document}